\documentclass{amsart}

\usepackage[utf8]{inputenc}

\usepackage{geometry}
\usepackage{amsthm}
\usepackage{amssymb}
\usepackage{amsmath}
\usepackage{comment}
\usepackage{thm-restate}
\usepackage{url} 
\usepackage{hyperref}
\usepackage[noabbrev,capitalise]{cleveref}
\usepackage{dsfont}

\usepackage{multicol}

\usepackage{enumitem}

\usepackage{color}
\usepackage[normalem]{ulem}

\usepackage{tikz}

\renewcommand{\v}{\textup{\textsf{v}}}
\newcommand{\e}{\textup{\textsf{e}}}
\newcommand{\tw}{\textup{\textsf{tw}}}
\renewcommand{\d}{\textup{\textsf{d}}}

\usepackage{caption}
\usepackage[labelformat=simple]{subcaption}

\theoremstyle{plain}
\newtheorem{thm}{Theorem}[section]
\newtheorem{lem}[thm]{Lemma}

\theoremstyle{definition}

\newcommand{\R}{\mathbb{R}}

\newcommand{\N}{\mathbb{N}}

\newcommand{\ad}{\overline {\textup{\textsf{d}}}}

\allowdisplaybreaks

\makeatletter
\@namedef{subjclassname@2020}{%
\textup{2020} Mathematics Subject Classification}
\makeatother

\begin{document}

\title{Finding dense minors using average degree}

\thanks{The first author is supported by the Institute for Basic Science (IBS-R029-C1). The second and fourth authors are supported by the Natural Sciences and Engineering Research Council of Canada (NSERC). Les deuxi\`eme et quatri\`eme auteurs sont supportés par le Conseil de recherches en sciences naturelles et en génie du Canada (CRSNG). The third author is funded by an ETH Z\"{u}rich Postdoctoral Fellowship.}


\author{Kevin Hendrey}
\address{Discrete Mathematics Group, Institute for Basic Science (IBS), Daejeon, South Korea}
\email{kevinhendrey@ibs.re.kr}
\urladdr{https://sites.google.com/view/kevinhendrey/}

\author{Sergey Norin}
\address{Department of Mathematics and Statistics, McGill University, Montr\'eal, Canada}
\email{snorin@math.mcgill.ca}
\urladdr{www.math.mcgill.ca/snorin/}

\author{Raphael Steiner}
\address{Institute of Theoretical Computer Science, Department of Computer Science, ETH Z\"{u}rich, Switzerland}
\email{raphaelmario.steiner@inf.ethz.ch}
\urladdr{https://sites.google.com/view/raphael-mario-steiner/}

\author{J\'er\'emie Turcotte}
\address{Department of Mathematics and Statistics, McGill University, Montr\'eal, Canada}
\email{mail@jeremieturcotte.com}
\urladdr{www.jeremieturcotte.com}

\begin{abstract}
	Motivated by Hadwiger's conjecture, we study the problem of finding the densest possible $t$-vertex minor in graphs of average degree at least $t-1$. We show that if $G$ has average degree at least $t-1$, it contains a minor on $t$ vertices with at least $(\sqrt{2}-1-o(1))\binom{t}{2}$ edges. We show that this cannot be improved beyond $\left(\frac{3}{4}+o(1)\right)\binom{t}{2}$. Finally,  for $t\leq 6$ we exactly determine the number of edges we are guaranteed to find in the densest $t$-vertex minor in graphs of average degree at least $t-1$.
\end{abstract}

\maketitle


\section{Introduction}

In this paper all graphs are simple and finite. We say a graph $H$ is a \emph{minor} of a graph $G$ if a graph isomorphic to $H$ can be obtained from a subgraph of $G$ by contracting edges (and removing any loops and parallel edges). A \emph{$k$-colouring} of a graph $G$ is an assignment of colours from $\{1,\dots, k\}$ to the vertices of $G$ such that adjacent vertices are assigned distinct colours. The \emph{chromatic number} $\chi(G)$ is the smallest integer $k$ such that $G$ admits a $k$-colouring.

Hadwiger \cite{hadwiger_uber_1943} conjectured that if $\chi(G) \geq t$, then $G$ contains $K_t$, the complete graph on $t$ vertices, as a minor. Hadwiger's conjecture is one of the most famous open problems in graph theory. The study of Hadwiger's conjecture has spawned a large number of variants, strengthenings and relaxations; see \cite{nash_hadwigers_2016} for a survey of this area. 

For instance, there has been much progress in recent years in attempting to find the smallest function $f(t)$ for which $\chi(G)\geq f(t)$ implies that $G$ contains a $K_t$ minor. 
The current best bound $f(t) = \Omega(t\log\log t)$ is due to Delcourt and Postle \cite{delcourt_reducing_2022}.

Another strategy is to approach Hadwiger's conjecture by relaxing the condition that the minor be complete. Seymour \cite{nash_hadwigers_2016,seymour_birs_2017} asks for which graphs $H$ on $t$ vertices does $\chi(G)\geq t$ guarantee $H$ as a minor. Kostochka \cite{kostochka_k_st_2014,kostochka_k_st_2010} as well as the second and fourth authors of this paper \cite{norin_limits_2023} have proved that this holds for large classes of bipartite graphs.

The second author and Seymour \cite{norin_dense_2022} study the related question of maximizing the edge density of $t$-vertex minors of $G$ if $\chi(G)\geq t$, that is finding a minor of $G$ on $t$ vertices with as many edges as possible. Unlike in the previous relaxation we allow the minor to change depending on $G$. It is show in \cite{norin_dense_2022} that if $G$ has $n$ vertices and independence number 2 (and so $\chi(G)\geq \left\lceil\frac{n}{2}\right\rceil$), then $G$ contains a minor on $\left\lceil\frac{n}{2}\right\rceil$ vertices and at least $(0.98688-o(1))\binom{\left\lceil\frac{n}{2}\right\rceil}{2}$ edges.

Letting both the number of vertices and edges of the minor vary, Nguyen \cite{nguyen_linear-sized_2022} showed that there exists $C>0$ such that if $\varepsilon\in (0,\frac{1}{256})$ and $\chi(G)\geq C t\log \log \left(\frac{1}{\varepsilon}\right)$, then $G$ contains a minor on $t$ vertices with at least $(1-\varepsilon)\binom{t}{2}$ edges. Setting $\varepsilon=\frac{1}{t^2}$ recovers the result of Delcourt and Postle mentioned above.

In this paper, we study the following strengthening of the question considered by the second author and Seymour: What is the densest $t$-vertex minor of $G$ if the average degree is at least $t-1$~? This is motivated by the well-known fact that $\chi(G)\geq t$ implies that $G$ contains a subgraph $G'$ with minimum degree at least $t-1$. 

It is a direct consequence of a result of Mader \cite{mader_homomorphieeigenschaften_1967,mader_homomorphiesatze_1968} (see \cref{lem:neighbourhoodmindegree} below) that every graph of average degree at least $t-1$ contains a minor on $t$ vertices with at least $\frac{1}{4}\binom{t}{2}$ edges.

Our main result is the following improvement, which we prove in \cref{sec:lower}. Let $\ad(G)$ denote the average degree of a graph $G$.

\begin{restatable}{thm}{mainthm}\label{thm:main}
	If $t\in \N$ and $G$ is a graph with average degree $\ad(G)\geq t$, then $G$ contains a minor on $t$ vertices with at least $\left(\sqrt 2-1-\frac{24}{t}\right)\binom{t}{2}$ edges.
\end{restatable}

Note that in \cref{thm:main}  we assume for convenience that $\ad(G)\geq t$ and not $\ad(G)\geq t-1$ as in the problem above. However, this only affects the lower order terms. 

In \cref{sec:upper}, we show that the constant in the previous result cannot be improved past $\frac{3}{4}$.

\begin{restatable}{thm}{upperthm}\label{thm:upper}
    For $t\in \N$, there exists a graph $G$ with average degree $\ad(G)\geq t$ such that $G$ does not contain a minor on $t$ vertices with more than $\left(\frac{3}{4}+o(1)\right)\binom{t}{2}$ edges.
\end{restatable}

In fact, we will describe a large class of graphs which satisfy the condition of \cref{thm:upper}.

Finally, in  \cref{thm:small}, proved in \cref{sec:small} we determine for small values of $t$ the exact number of edges we are guaranteed to find in the densest $t$-vertex minor.

\begin{restatable}{thm}{smallthm}\label{thm:small}
    If $2\leq t\leq 6$ is an integer and $G$ is a graph with average degree $\ad(G)\geq t-1$, then $G$ contains a minor on $t$ vertices with at least 
    \begin{itemize}
        \item 1 edge if $t=2$,
        \item 3 edges if $t=3$,
        \item 5 edges if $t=4$,
        \item 8 edges if $t=5$, and
        \item 11 edges it $t=6$.
    \end{itemize}
    Furthermore, none of these values can be improved.
\end{restatable}

\subsection{Notation}
Let $G$ be a graph. We denote by $V(G)$ the set of vertices of $G$ and $E(G)\subseteq \binom{V(G)}{2}$ the set of edges of $G$. We will write $\v(G)=|V(G)|$ for the number of vertices of $G$ and $\e(G)=|E(G)|$ for the number of edges of $G$. If $u\in V(G)$, we write $N_G(u)$ for the (open) neighbourhood of $u$, $N_G[u]=N_G(u)\cup\{u\}$ for the closed neighbourhood of $u$, and $\deg_G(u)=|N_G(u)|$ for the degree of $u$ in $G$. We denote the minimum degree of $G$ by $\delta(G)=\min_{u\in V(G)}\deg_G(u)$ and the average degree of $G$ by $\ad(G)=\frac{\sum_{u\in V(G)}\deg_G(u)}{\v(G)}$. If $X\subseteq V(G)$, we write $G[X]$ for the subgraph of $G$ induced by $X$, and $G-X=G[V(G)\setminus X]$ for the subgraph obtained by removing the vertices in $X$; if $X=\{u\}$ we will write $G-u$ for $G-X$. If $e\in E(G)$, we write $G-e$ for the graph obtained by removing $e$, and $G/e$ for the graph obtained from $G$ by contracting the edge $e$ (and removing any resulting loops or duplicate edges); in particular $G/e$ is a minor of $G$. If $Z$ is a real-valued random variable, we write $\mathbb E[Z]$ for the expected value of $Z$.

\section{Lower bound}\label{sec:lower}

In this section, we prove \cref{thm:main}. 

The following result of Mader \cite[pages 1--2]{mader_homomorphieeigenschaften_1967} will allow us to get a lower bound on the minimum degree in the neighbourhood of each vertex, by taking a minor of our graph. We include a short proof for the sake of completeness, since the paper~\cite{mader_homomorphieeigenschaften_1967} is only available in German.

\begin{lem}\label{lem:neighbourhoodmindegree}
	If $G$ is a graph, then $G$ contains a minor $H$ such that $\ad(H)\geq \ad(G)$ and $\delta\left(H[N[u]]\right)> \frac{\ad(G)}{2}$ for every $u\in V(H)$.
\end{lem}
\begin{proof}
	Let $H$ be a minor of $G$ such that $\ad(H)\geq \ad(G)$ which minimizes $\v(H)$. Suppose for a contradiction that $H$ does not respect the  statement, i.e. there is some $u\in V(H)$ such that $\delta\left(H[N[u]]\right)\leq\frac{\ad(G)}{2}$.
	
	If $\deg_{H}(u)=0$, then it is direct that $\ad(H-u)\geq \ad(H)$, which contradicts the minimality of $H$. Hence, we may suppose that $u$ has at least 1 neighbour.
	
	Given that the degree of $u$ in $H[N[u]]$ is as least as large as the degree in $H[N[u]]$ of every other vertex in $N[u]$, there exists $v\in N(u)$ such that $\deg_{H[N[u]]}(v)\leq\frac{\ad(G)}{2}\leq \frac{\ad(H)}{2}$. In other words, $|N[u]\cap N(v)|\leq\frac{\ad(H)}{2}$. Then,
	$$\ad(H/uv)=\frac{2\e(H/uv)}{\v(H/uv)}=\frac{2(\e(H)-|N[u]\cap N(v)|)}{\v(H)-1}\geq \frac{2\e(H)-\ad(H)}{\v(H)-1}=\frac{\v(H)\cdot \ad(H)-\ad(H)}{\v(H)-1}=\ad(H)$$
	which is a contradiction to the minimality of $H$.
\end{proof}

The next easy lemma tells us when removing vertices does not decrease average degree.

\begin{lem}\label{lem:removingsubset}
	If $G$ is a graph and $X\subsetneq V(G)$ is such that exactly $M$ edges of $G$ have at least one end in $X$ and $M\leq \frac{\ad(G)\cdot|X|}{2}$, then $\ad(G-X)\geq \ad(G)$.
\end{lem}
\begin{proof}We may compute directly that
	\begin{align*}
		\ad(G-X)=\frac{2\e(G-X)}{\v(G-X)}=\frac{2\left(\e(G)-M\right)}{\v(G)-|X|}\geq\frac{2\e(G)-\ad(G)\cdot|X|}{\v(G)-|X|}=\frac{\v(G)\cdot \ad(G)-\ad(G)\cdot|X|}{\v(G)-|X|}=\ad(G).
	\end{align*}
\end{proof}

The next lemma allows us to extract a dense subgraph on $t$ vertices; this is a standard application of the first moment method.

\begin{lem}\label{lem:fullrandom}
	If $t\in \N$ and $G$ is a graph with $\v(G)\geq t$, then $G$ contains a subgraph on $t$ vertices with at least $\frac{\ad(G)}{\v(G)}\binom{t}{2}$ edges.
\end{lem}
\begin{proof}
	The statement is trivial when $t=1$, so we may assume that $t\geq 2$. Let $Z$ be a uniformly random subset of $V(G)$ of size $t$. Given $uv\in E(G)$, the probability that $uv$ is an edge of $G[Z]$ 
 is $\frac{\binom{\v(G)-2}{t-2}}{\binom{\v(G)}{t}}=\frac{t(t-1)}{\v(G)(\v(G)-1)}$. As $\e(G)=\frac{\v(G)\cdot \ad(G)}{2}$, we  have 
	\begin{align*}
		\mathbb E[\e(G[Z])]
		=\frac{\v(G)\cdot \ad(G)}{2}\cdot \frac{t(t-1)}{\v(G)(\v(G)-1)}
		=\frac{\ad(G)}{\v(G)-1}\binom{t}{2}\geq\frac{\ad(G)}{\v(G)}\binom{t}{2}.
	\end{align*}
	Hence, there exists at least one choice of $Z$ such that $\e(G[Z])\geq\frac{\ad(G)}{\v(G)}\binom{t}{2}$. Hence the statement holds for $G[Z]$.
\end{proof}

In the next lemma, we find a dense subgraph on $t$ vertices by extending an already dense, but not large enough, set of vertices $X$ to a set of size $t$ by sampling the remaining vertices in another set $Y$. We will apply this lemma  when $X$ is a union of closed neighbourhoods and $Y$ a closed neighbourhood (or vice versa), the conditions on the minimum degrees of the induced subgraphs on these sets will come from \cref{lem:neighbourhoodmindegree}.

\begin{lem}\label{lem:partialrandom}
	If $t\in \N$, $G$ is a graph and $X,Y\subseteq V(G)$ are such that $|X|\leq t$, $|X\cup Y|\geq t$ and $\delta(G[X]),\delta(G[Y])\geq \frac{t}{2}$, then there is a subgraph of $G$ on $t$ vertices with at least $\left(\frac{1}{2}\left(x+\frac{(1-x)^2}{y}\right)-\frac{1}{t}\right)\binom{t}{2}$ edges, where $x=\frac{|X|}{t}$ and $y=\frac{|Y|}{t}$.
\end{lem}
\begin{proof}
    If $|X|=t$, the statement follows directly by considering the $t$-vertex subgraph $G[X]$, which contains at least $\frac{1}{2}\cdot|X|\cdot\delta(G[X]) \ge \frac{1}{2}\cdot t\cdot \frac{t}{2} \ge \frac{1}{2}\binom{t}{2}$ edges. Hence, we may suppose that $|X|\leq t-1$, and as a consequence that $|Y\setminus X|\geq 1$.

	Let $Y'$ be a uniformly random subset of $Y\setminus X$ of size $t-|X|$ and set $Z=X\cup Y'$, which is possible given that $|X\cup Y|\geq t$ implies $t-|X|\leq |Y\setminus X|$.
	
	The number of edges with both ends in $X$ is at least $\frac{1}{2}\cdot|X|\cdot \delta(G[X])\geq\frac{t|X|}{4}$. The number of edges with both ends in $Y\setminus X$ is $\frac{1}{2}\sum_{v\in Y\setminus X} |N(v)\cap (Y\setminus X)|$ and the number of edges between $X$ and $Y\setminus X$ is $\sum_{v\in Y\setminus X} |N(v)\cap X|$, so the number of edges with both ends in $X\cup Y$ and at least one end in $Y\setminus X$ is
    \begin{align*}
        \frac{1}{2}\sum_{v\in Y\setminus X} |N(v)\cap (Y\setminus X)|+\sum_{v\in Y\setminus X} |N(v)\cap X|
        &\geq \frac{1}{2}\sum_{v\in Y\setminus X}|N(v)\cap (X \cup Y)|\\
        &\geq\frac{1}{2} |Y\setminus X|\cdot \delta(G[X\cup Y])\geq\frac{t|Y\setminus X|}{4}.
    \end{align*}
	
	Suppose $uv$ is an edge of $G[X\cup Y]$ with at least one end in $Y\setminus X$, say $u\in Y\setminus X$. Then, $uv$ is an edge in $G[Z]$ if and only if  $\{u,v\}\subseteq Z$. If $v\in Y\setminus X$, then the probability that $uv$ is an edge is the probability that $\{u,v\}\subseteq Y'$, which is $\frac{\binom{|Y\setminus X|-2}{t-|X|-2}}{\binom{|Y\setminus X|}{t-|X|}}=\frac{(t-|X|)(t-|X|-1)}{|Y\setminus X|\left(|Y\setminus X|-1\right)}\geq\frac{(t-|X|)(t-|X|-1)}{|Y\setminus X|^2}$. If $v\in X$, then the probability that $uv$ is an edge is the probability that $u\in Y'$, which is $\frac{t-|X|}{|Y\setminus X|}\geq \frac{(t-|X|)(t-|X|-1)}{|Y\setminus X|^2}$. Here we use the fact that $t-|X|\leq |Y\setminus X|$
	Hence, we have 
	\begin{align*}
		\mathbb E[\e(G[Z])]
		&\geq \frac{t|X|}{4}+\frac{t|Y\setminus X|}{4}\cdot \frac{(t-|X|)(t-|X|-1)}{|Y\setminus X|^2}\\
		&\geq\frac{t|X|}{4}+\frac{t}{4}\cdot \frac{(t-|X|)(t-|X|-1)}{|Y|}\\
		&\geq\frac{(t-1)\cdot tx}{4}+\frac{t-1}{4}\cdot \frac{(t-tx)(t-tx-1)}{ty}\\
		&=\frac{1}{2}\left(x+\frac{(1-x)(1-x-\frac{1}{t})}{y}\right)\binom{t}{2}\\
		&\geq\left(\frac{1}{2}\left(x+\frac{(1-x)^2}{y}\right)-\frac{1}{t}\right)\binom{t}{2},
	\end{align*}
where in the last step we used that $1-x \le y$. Hence, there is at least one choice of $Y'$ such that $G[Z]=G[X \cup Y']$ has $t$ vertices and at least $\left(\frac{1}{2}\left(x+\frac{(1-x)^2}{y}\right)-\frac{1}{t}\right)\binom{t}{2}$ edges, as desired.
\end{proof}

The next lemma finds a dense subgraph on $t$ vertices given a dense set $X$ if the vertices outside of $X$ all have sufficiently large degree. Contrary to the previous lemma, here the set $X$ is not extended to a set of size $t$, instead its properties will allow us to show that $G$ is itself not too large, and so a good candidate to apply \cref{lem:fullrandom}.

\begin{lem}\label{lem:fullgraphcase}
	Let $t,c\in \N$, $\lambda>1+\frac{3}{t}$ and let $G$ be a graph with $\ad(G)\geq t$ and such that $\ad(H)\leq t$ for all non-null proper subgraphs $H$ of $G$. If $\emptyset \neq X\subsetneq V(G)$ is such that $\delta(G[X])\geq \frac{t}{2}$ and $\deg_G(u)>\lambda t-1$ for all $u\in V(G)\setminus X$, then $G$ contains a subgraph on $t$ vertices with at least $\frac{(\lambda-1-\frac{3}{t})t}{(\lambda-\frac{3}{4})|X|}\binom{t}{2}$ edges.
\end{lem}
\begin{proof}
	First note that $\ad(G-X)<t\leq \ad(G)$. Let $M$ be the number of edges with at least one end in $X$. \cref{lem:removingsubset} implies that $M>\frac{\ad(G)\cdot|X|}{2}\geq \frac{t|X|}{2}$. In particular, we then have that the sum of degrees of vertices in $X$ is
	\begin{align*}
		\sum_{v\in X} \deg_G(v)
		&=\sum_{v\in X} |N(v)\cap X|+\sum_{v\in X} |N(v)\setminus X|\\
        &=\left(\frac{1}{2}\sum_{v\in X} |N(v)\cap X|+\sum_{v\in X} |N(v)\setminus X|\right)+\frac{1}{2}\sum_{v\in X} |N(v)\cap X|\\
		&= M+\frac{1}{2} \sum_{v\in X} |N(v)\cap X|\\
		&>\frac{t|X|}{2}+\frac{|X|\delta(G[X])}{2}\\
		&\geq\frac{3t|X|}{4}.
	\end{align*}
	Furthermore, we have $$\sum_{v\in V(G)\setminus X}\deg_G(v)\geq (\v(G)-|X|)(\lambda t-1).$$

    Let $u$ be any vertex of $G$. As $\ad(G-u)<t$,
    $$\ad(G)=\frac{2\e(G)}{\v(G)}=\frac{2(\e(G-u)+\deg_G(u))}{\v(G)}=\frac{\v(G-u)\ad(G-u)+2\deg_G(u)}{\v(G)}<\ad(G-u)+2\leq t+2.$$
    
    Hence, 
	\begin{align*}
		\v(G)(t+2)> \v(G)\cdot \ad(G)=\sum_{v\in V(G)}\deg_G(v)
		&\geq\frac{3t|X|}{4}+(\v(G)-|X|)(\lambda t-1).
	\end{align*}
	Rearranging yields that
	$$\v(G)< \frac{(\lambda-\frac{3}{4}-\frac{1}{t})|X|}{\lambda-1-\frac{3}{t}}<\frac{(\lambda-\frac{3}{4})|X|}{\lambda-1-\frac{3}{t}}.$$
	
	Since $\ad(G)\geq t$, we in particular have that $\v(G)\geq t$. Finally, to get the desired subgraph we apply \cref{lem:fullrandom} to $G$ to get a subgraph on $t$ vertices with at least
	$$\frac{\ad(G)}{\v(G)}\binom{t}{2}\geq\frac{(\lambda-1-\frac{3}{t})t}{(\lambda-\frac{3}{4})|X|}\binom{t}{2}$$ edges.
\end{proof}

The next lemma is a core element of our proof, which we summarize here. We want to find a set $X$ of vertices such that $G[X]$ has minimum degree at least $\frac{t}{2}$ which has order as close as possible to $t$. We will be taking $X$ to be a union of closed neighbourhoods in our graph (as, by \cref{lem:neighbourhoodmindegree}, we will be able to assume that closed neighbourhoods have this minimum degree); we can construct this set by sequentially adding neighbourhoods of vertices. The closer $|X|$ is to $t$, the denser a subgraph of $G$ of order $t$ we will be able to find; when $|X|>t$ we can use \cref{lem:fullrandom} to sample a dense subset of size exactly $t$ and when $|X|<t$ we can use either \cref{lem:partialrandom} or \cref{lem:fullgraphcase} (depending on the degrees of vertices outside of $X$). In practice, we will introduce some tolerance and attempt to find $X$ such that $\alpha t\leq |X|\leq \beta t$ (this is Case \ref{case1} in the following lemma). The first way of failing is when constructing $X$, at some point the size is smaller than $\alpha t$, but the size jumps over $\beta t$ when adding any other neighbourhood of size smaller than $\alpha t$; this is Case \ref{case3}, to which we will apply \cref{lem:partialrandom}. The other way of failing is to run out of vertices of small degree, in which case we are in Case \ref{case2}. In this case, the fact that all remaining vertices have large degree will allow us to apply \cref{lem:fullgraphcase}. The parameters $\alpha,\beta$ will later be chosen to optimize the trade-offs between these cases.

\begin{lem}\label{lem:smallsubgraph}
	If $t\in \N$, $G$ is a graph such that $\delta(G[N[u]])\geq \frac{t}{2}$ for every $u\in V(G)$, and $\frac{1}{2}\leq\alpha<\beta\in \R$, then either
	\begin{enumerate}[label=(\arabic*)]
		\item\label{case1} there exists $X\subseteq V(G)$ such that $\alpha t\leq |X|\leq \beta t$ and $\delta(G[X])\geq \frac{t}{2}$,
		\item\label{case2} there exists $X\subseteq V(G)$ such that $\frac{t}{2}\leq |X|<\alpha t$, $\delta(G[X])\geq \frac{t}{2}$, and $\deg_G(u)> \beta t-1$ for every $u\in V(G)\setminus X$, or
		\item\label{case3} there exist $X,Y\subseteq V(G)$ such that $|X|,|Y|<\alpha t$, $|X\cup Y|>\beta t$ and $\delta(G[X]),\delta(G[Y])\geq \frac{t}{2}$.
	\end{enumerate}
\end{lem}
\begin{proof}
	If there exists $u\in V(G)$ such that $\alpha t-1\leq \deg_G(u)\leq \beta t-1$, then setting $X=N[u]$ we are in Case \ref{case1}. Hence, we may assume that for every $u\in V(G)$, $\deg_G(u)<\alpha t-1$ or $\deg_G(u)>\beta t-1$.
	
	Let $A=\bigcup\{N[u] : u\in V(G), \deg(u)<\alpha t-1\}$. Note that since $\delta(G[N[u]])\ge \frac{t}{2}$ for every $u \in V(G)$, we also have $\delta(G[A])\ge \frac{t}{2}$ and thus $|A|>\frac{t}{2}$.
 
 If $A$ is such that $|A|<\alpha t$, then we are in Case \ref{case2} for $X=A$, since every vertex not in $A$ has degree at least $\alpha t-1$, and hence greater than $\beta t-1$ by the assumption above.
	
	Otherwise, we can find $B=\bigcup_{i=1}^k N[x_i]$, for $x_1,\dots,x_k\in V(G)$ all of degree smaller than $\alpha t-1$, such that $|B|\geq \alpha t$. Pick such a set which minimizes $k$. Again, note that $\delta(G[B])\ge \frac{t}{2}$. If $|B|\leq \beta t$, we are in Case \ref{case1} with $X=B$. Hence suppose that $|B|>\beta t$.
	
	Note that necessarily $k\geq 2$, as if $B=N[x_1]$, we have $|B|=\deg_G(x_1)+1<\alpha t$. By minimality of $k$, $|\bigcup_{i=1}^{k-1} N[x_i]|<\alpha  t$. Hence, Case \ref{case3} holds for $X=\bigcup_{i=1}^{k-1} N[x_i]$ and $Y=N[x_k]$.
\end{proof}

Finally, we are ready to derive
\cref{thm:main}, which we restate for convenience, from the above lemmas.

\mainthm*

\begin{proof}
	The statement is trivial for $t\leq 24$, so assume $t\geq 25$. We may suppose that $\ad(G)$ has no proper minor such that $\ad(G)\geq t$; in particular, $G$ contains no proper minor $H$ such that $\ad(H)\geq \ad(G)$. By \cref{lem:neighbourhoodmindegree} we have that $\delta(G[N[u]])\geq\frac{t}{2}$ for every $u\in V(G)$. 

    Let $\alpha=\frac{4}{5}$, $\beta=\nu=\frac{6}{5}$ and let $\gamma= \sqrt 2-1-\frac{24}{t}$. We wish to prove that $G$ contains a subgraph on $t$ vertices with at least $\gamma\binom{t}{2}$ edges. As $1/2 < \alpha < \beta$ we can apply \cref{lem:smallsubgraph} to $G$ and consider three cases depending on the outcome of  \cref{lem:smallsubgraph} which holds.
	
	First suppose we are in Case \ref{case1}, i.e.   there exists $X\subseteq V(G)$ such that $\alpha t\leq |X|\leq \beta t$ and $\delta(G[X])\geq \frac{t}{2}$. There are two subcases here. The first is $t\leq |X|\leq \beta t$. \cref{lem:fullrandom} applied to $G[X]$ then ensures that $G$ contains a subgraph on $t$ vertices with at least $$\frac{\ad(G[X])}{\v(G[X])}\binom{t}{2}\geq \frac{\frac{t}{2}}{\beta t}\binom{t}{2}=\frac{1}{2\beta}\binom{t}{2} = \frac {5}{12} \binom{t}{2} >\gamma\binom{t}{2}$$ edges. 
	
	The other subcase is $\alpha t\leq |X|\leq t$. In the following, we assume without loss of generality that $X$ is chosen of maximum size subject to satisfying the conditions of Case~(1) in Lemma~\ref{lem:smallsubgraph} as well as $|X|\leq t$. In particular, this implies that for every vertex $u \in V(G)\setminus X$, we have $|X \cup N[u]|>t$, for otherwise we could replace $X$ with $X \cup N[u]$, contradicting the maximality assumption. 
 
 Let us consider two possibilities. The first is that there exists $u\in V(G)\setminus X$ such that $\deg_G(u)\leq \nu t-1$. Let $Y=N[u]$ and write $|X|=xt$ and $|Y|=yt$. By \cref{lem:partialrandom} there exists a subgraph of $G$ on $t$ vertices with at least $$\left(\frac{1}{2}\left(x+\frac{(1-x)^2}{y}\right)-\frac{1}{t}\right)\binom{t}{2}\geq\left(\frac{1}{2}\left(x+\frac{(1-x)^2}{\nu}\right)-\frac{1}{t}\right)\binom{t}{2}$$ edges. Hence, the theorem holds in  this case as
	\begin{equation*}
		\gamma\leq \min_{\substack{\alpha\leq x\leq 1}} \frac{1}{2}\left(x+\frac{(1-x)^2}{\nu}\right)-\frac{1}{t} =  \frac{1}{2}\left(\alpha+\frac{(1-\alpha)^2}{\nu}\right)-\frac{1}{t}  = \frac{5}{12}-\frac{1}{t}.
	\end{equation*}
	
	Otherwise, $\deg_G(u)> \nu t-1$ for every $u\in V(G)\setminus X$. By \cref{lem:fullgraphcase} $G$ contains a subgraph on $t$ vertices with at least $$\frac{(\nu-1-\frac{3}{t})t}{(\nu-\frac{3}{4})|X|}\binom{t}{2}\geq\frac{(\nu-1-\frac{3}{t})t}{(\nu-\frac{3}{4}) t}\binom{t}{2}>\left(\frac{\nu-1}{\nu-\frac{3}{4}}-\frac{12}{t}\right)\binom{t}{2} = \left(\frac{4}{9}-\frac{12}{t}\right)\binom{t}{2} >  \gamma\binom{t}{2}$$
	edges, where the second inequality holds as  $\nu>1$. This finishes the proof in the case that outcome \ref{case1} of \cref{lem:smallsubgraph} holds.

	Now suppose that outcome \ref{case2} of \cref{lem:smallsubgraph} holds, i.e.  there exists $X\subseteq V(G)$ such that $\frac{t}{2}\leq |X|<\alpha t$, $\delta(G[X])\geq \frac{t}{2}$, and $\deg_G(u)> \beta t-1$ for every $u\in V(G)\setminus X$.  By \cref{lem:fullgraphcase}, $G$ contains a subgraph on $t$ vertices with at least $$\frac{(\beta-1-\frac{3}{t})t}{(\beta-\frac{3}{4})|X|}\binom{t}{2}>\frac{(\beta-1-\frac{3}{t})t}{(\beta-\frac{3}{4}) t}\binom{t}{2}>\left(\frac{\beta-1}{(\beta-\frac{3}{4})}-\frac{24}{t}\right)\binom{t}{2} = \left(\frac{4}{9}-\frac{24}{t}\right)\binom{t}{2} >  \gamma\binom{t}{2}$$
	edges, where the second inequality uses $\beta>1$. 
	
	Finally, suppose that outcome \ref{case3} of \cref{lem:smallsubgraph} holds, i.e. there exist $X,Y\subseteq V(G)$ such that $|X|,|Y|<\alpha t$, $|X\cup Y|>\beta t$ and $\delta(G[X]),\delta(G[Y])\geq \frac{t}{2}$. Without loss of generality suppose $|X|\geq |Y|$ and let $|X|=xt$ and $|Y|=yt$. By \cref{lem:partialrandom} 
 there exists a subgraph of $G$ on $t$ vertices with at least $$\left(\frac{1}{2}\left(x+\frac{(1-x)^2}{y}\right)-\frac{1}{t}\right)\binom{t}{2} \geq \left(\frac{1}{2}\left(x+\frac{(1-x)^2}{x}\right)-\frac{1}{t}\right)\binom{t}{2} $$ edges. Hence, it suffices to show that 
 $$ \frac{1}{2}\left(x+\frac{(1-x)^2}{x}\right) \geq \sqrt{2}-1.$$
This last inequality simplifies to $\left(\sqrt{2} x -1\right)^2 \geq 0$ for $x > 0 $ and hence holds for all such $x$, as desired.
\end{proof}

\section{Upper bound}\label{sec:upper}

In this section, we prove \cref{thm:upper}. We first need the following definitions.

For $k\in \N$, \emph{$k$-trees} are the graph family defined in a recursive manner as follows:
\begin{itemize}
    \item The complete graph $K_{k+1}$ is a $k$-tree.
    \item If $G$ is a $k$-tree and $C \subseteq V(G)$ is a clique in $G$ with $|C|=k$, then the graph obtained from $G$ by adding a new vertex with neighbourhood $C$ is also a $k$-tree.
\end{itemize}
It follows readily from this definition that for any $k$-tree $G$, $\e(G)=\binom{k}{2}+k(\v(G)-k)$. Furthermore, every minor of a $k$-tree with at least $k+1$ vertices is also a spanning subgraph of a $k$-tree. Indeed, the treewidth $\tw(G)$ of a graph $G$ with at least $k+1$ vertices is at most $k$ if and only if $G$ is a spanning subgraph of a $k$-tree \cite{scheffler_graphs_1986,wimer_linear_1987}, and it is well-known that treewidth is minor-monotone (that is, $\tw(G')\le \tw(G)$ for every minor $G'$ of a graph $G$). Sufficiently large $k$-trees, with appropriately chosen parameter $k$, will be our first candidates for \cref{thm:upper}.

Let $S_r=K_{1,r}$ be the star graph with $r$ leaves. We define the graph $S_{k,r,s}$ as the graph obtained from $S_r$ by replacing every leaf by cliques $A_1,\dots,A_r$ on $s$ vertices and replacing the central vertex by a clique $C$ on $k$ vertices. In particular, every vertex of $C$ is adjacent to every other vertex in the graph. Such vertices are said to be \emph{universal}.

These graphs, with appropriately chosen parameters, will also be candidates for \cref{thm:upper}. First note that $$\ad(S_{k,r,s})=\frac{2\left(\binom{k}{2}+r\binom{s}{2}+k r s\right)}{k+rs}.$$

Given graphs $G$ and $H$, we say the collection $(B_u)_{u\in V(H)}$ of pairwise disjoint non-empty subsets of $V(G)$ is a \emph{model} of $H$ in $G$ if $G[B_u]$ is connected for every $u\in V(H)$ and $G$ contains at least one edge between $B_u$ and $B_v$ if $uv\in E(H)$. It is easy to see that there exists a model of $H$ in $G$ if and only if $H$ is a minor of $G$. It is also direct that if $|B_u|=1$ for every $u\in V(H)$, then $G$ contains a subgraph isomorphic to $H$ (precisely, on vertex set $\bigcup_{u\in V(H)}B_u$).

We now show that with these graphs, we may restrict ourselves to finding a dense subgraph on $t$ vertices, which is simpler than finding minors.

\begin{lem}\label{lem:minortosubgraph}
    If $k,r,s\in \N$ and $H$ is a minor of $S_{k,r,s}$, then $S_{k,r,s}$ has a subgraph isomorphic to $H$.
\end{lem}
\begin{proof}
    Let $\mathcal B=(B_u)_{u\in V(H)}$ be a model of $H$ in $G$ which minimizes $\sum_{u\in V(H)}|B_u|$. If $|B_u|=1$ for every $u\in V(H)$, then we are done by the above remark. Hence, assume $|B_v|\geq 2$ for some $v\in V(H)$.

    If $B_v\subseteq A_i$ for some $1\leq i\leq r$, let $x\in B_v$. If $B_v\cap C\neq \emptyset$, let $x\in C\cap B_v$. Given the structure of $S_{k,r,s}$ and that the subgraph induced by $B_v$ is connected, these are the only two possible cases. Let $B_v'=\{x\}$ and $B_u'=B_u$ for $u\in V(H)\setminus\{v\}$. 
    
    It is easy to verify that in both of these cases, if $w\in V(G)\setminus B_v$ is adjacent to at least one vertex of $B_v$, then $w$ is adjacent to $x$. Hence, $(B_u')_{u\in V(H)}$ is a model of $H$ in $G$ which contradicts the minimality of $\mathcal B$.
\end{proof}

We now compute an upper bound on the density of $t$-vertex subgraphs of $S_{k,r,s}$.

\begin{lem}\label{lem:numberedges}
    If $k,r,s,t\in \N$ are such that $k+rs\geq t\geq k$, then $S_{k,r,s}$ does not contain a subgraph on $t$ vertices with more than
    \begin{equation}\label{eq:upperedges}
        f(k,s,t)=\binom{k}{2}+k(t-k)+\left\lfloor\frac{t-k}{s}\right\rfloor\binom{s}{2}+\binom{t-k-\left\lfloor\frac{t-k}{s}\right\rfloor s}{2}
    \end{equation}
    edges.
\end{lem}

\begin{proof}
    Let $X\subseteq V(S_{k,r,s})$ such that $|X|=t$ which maximizes the number of edges in $G[X]$ with first priority and then, subject to $\e(G[X])$ being maximum, maximizes $|X\cap C|$ with second priority. This is possible given that $\v(S_{k,r,s})=k+rs\geq t$. Our goal is thus to upper bound $\e(G[X])$.
    
    We first claim that $C\subseteq X$. Suppose to contrary that there exists $c\in C\setminus X$. Given that $|X|=t\geq k=|C|>|C\setminus\{c\}|$, this implies there exists $x\in X\setminus C$. Let $X_0=X\setminus \{x\}$ and $X'=X_0\cup \{c\}$.
    Then,
    $$\e(G[X])
    =\e\left(G\left[X_0\right]\right)+\left|N(x)\cap X_0\right|
    \leq \e\left(G\left[X_0\right]\right)+\left|N(c)\cap X_0\right|
    = \e(G[X']), $$
    where the inequality follows from the fact that $c$ is a universal vertex of $G$. Since $|X' \cap C|>|X \cap C|$, this contradicts our choice of $X$. Hence, we have $C\subseteq X$.

    The number of edges in $G[C]$ is $\binom{k}{2}$. Given that $|X\setminus C|=t-k$ and every vertex in $X\setminus C$ is connected to all $k$ vertices in $C$, the number of edges between $C$ and $X\setminus C$ is $k(t-k)$.
    
    For $1\leq i\leq r$, let $a_i=|X \cap A_i|$. In particular, $\sum_{i=1}^r{a_i}=t-k$ and for every $i$ we have $0\leq a_i\leq s$. Then $\e(G[X\setminus C])=\sum_{i=1}^r\binom{a_i}{2}$. Suppose there exist distinct $1\leq i,j\leq r$ such that $0<a_i,a_j<s$. Without loss of generality, suppose $a_i\geq a_j$. Under this assumption, it is easy to verify that $\binom{a_i}{2}+\binom{a_j}{2}< \binom{a_i+1}{2}+\binom{a_j-1}{2}$, and so by choosing one more vertex in $A_i$ and one fewer vertex in $A_j$ we could obtain a subgraph with more edges, contradicting the maximality of $\e(G[X])$. Hence $a_i\in \{0,s\}$ for all except at most one index $i \in \{1,\ldots,r\}$. It then follows from $\sum_{i=1}^r{a_i}=t-k$ that $a_i=s$ for exactly $\left\lfloor\frac{t-k}{s}\right\rfloor$ choices of $i$, and the possible remaining non-empty set of the form $X \cap A_i$ contains exactly $(t-k)-\left\lfloor\frac{t-k}{s}\right\rfloor s$ vertices. Hence, $G[X\setminus C]$ contains exactly $\left\lfloor\frac{t-k}{s}\right\rfloor\binom{s}{2}+\binom{t-k-\left\lfloor\frac{t-k}{s}\right\rfloor s}{2}$ edges. This concludes the proof of the lemma.
\end{proof}

We may now prove \cref{thm:upper}, which we restate for convenience.

\upperthm*

\begin{proof}
    We prove the theorem in 2 ways.

    \vskip 3mm
    \noindent\textbf{Using $k$-trees}
    
    In order to prove the theorem, consider any $k(t)$-tree $G$, where $k(t)=\left(\frac{1}{2}+o(1)\right)t>\frac{t}{2}$, and for which $\v(G)$ is sufficiently large (as a function of $k(t)$) such that
    $$\ad(G)=2\frac{\binom{k(t)}{2}+k(t)(\v(G)-k(t))}{\v(G)}=2k(t)-\frac{(k(t)+1)k(t)}{\v(G)}> t.$$

    Let $G'$ be any minor of $G$ on $t$ vertices. As noted earlier, $G'$ must be a spanning subgraph of some $k(t)$-tree. Hence,
    $$\e(G')\leq \binom{k(t)}{2}+k(t)(t-k(t))=\left(\frac{1}{2}+o(1)\right)^2\frac{t^2}{2}+\left(\frac{1}{2}+o(1)\right)^2t^2=\left(\frac{3}{4}+o(1)\right)\binom{t}{2},$$
    as desired.

    See \cref{fig:ksumexample} for an example of such a graph.

    \begin{figure}
         \centering
         \begin{tikzpicture}[scale=0.75,
			dot/.style = {circle, fill, minimum size=#1,
			inner sep=0pt, outer sep=0pt},
			dot/.default = 4pt]
            
            \node[dot] (1) at (1,0) {};
            \node[dot] (2) at (2,0) {};
            \node[dot] (3) at (3,0) {};
            \node[dot] (4) at (4,0) {};
            \node[dot] (5) at (5,0) {};
            \node[dot] (6) at (6,0) {};
            \node[dot] (7) at (7,0) {};
            \node[dot] (8) at (8,0) {};
            \node[dot] (9) at (9,0) {};
            \node[dot] (10) at (10,0) {};
            \node[dot] (11) at (11,0) {};
            \node[dot] (12) at (12,0) {};
            \node[dot] (13) at (13,0) {};
            \node[dot] (14) at (14,0) {};
            \node[dot] (15) at (15,0) {};
            \node[dot] (16) at (16,0) {};
            \node[dot] (17) at (17,0) {};
            \node[dot] (18) at (18,0) {};
            \node[dot] (19) at (19,0) {};
            \node[dot] (20) at (20,0) {};
            \node[dot] (21) at (21,0) {};
   
			\draw (1) to (2);
            \draw (2) to (3);
            \draw (3) to (4);
            \draw (4) to (5);
            \draw (5) to (6);
            \draw (6) to (7);
            \draw (7) to (8);
            \draw (8) to (9);
            \draw (9) to (10);
            \draw (10) to (11);
            \draw (11) to (12);
            \draw (12) to (13);
            \draw (13) to (14);
            \draw (14) to (15);
            \draw (15) to (16);
            \draw (16) to (17);
            \draw (17) to (18);
            \draw (18) to (19);
            \draw (19) to (20);
            \draw (20) to (21);

            \draw[bend left] (1) to (3);
            \draw[bend left] (1) to (4);
            \draw[bend left] (1) to (5);
            \draw[bend left] (1) to (6);

            \draw[bend right] (2) to (4);
            \draw[bend right] (2) to (5);
            \draw[bend right] (2) to (6);
            \draw[bend right] (2) to (7);

            \draw[bend left] (3) to (5);
            \draw[bend left] (3) to (6);
            \draw[bend left] (3) to (7);
            \draw[bend left] (3) to (8);

            \draw[bend right] (4) to (6);
            \draw[bend right] (4) to (7);
            \draw[bend right] (4) to (8);
            \draw[bend right] (4) to (9);

            \draw[bend left] (5) to (7);
            \draw[bend left] (5) to (8);
            \draw[bend left] (5) to (9);
            \draw[bend left] (5) to (10);

            \draw[bend right] (6) to (8);
            \draw[bend right] (6) to (9);
            \draw[bend right] (6) to (10);
            \draw[bend right] (6) to (11);

            \draw[bend left] (7) to (9);
            \draw[bend left] (7) to (10);
            \draw[bend left] (7) to (11);
            \draw[bend left] (7) to (12);

            \draw[bend right] (8) to (10);
            \draw[bend right] (8) to (11);
            \draw[bend right] (8) to (12);
            \draw[bend right] (8) to (13);

            \draw[bend left] (9) to (11);
            \draw[bend left] (9) to (12);
            \draw[bend left] (9) to (13);
            \draw[bend left] (9) to (14);

            \draw[bend right] (10) to (12);
            \draw[bend right] (10) to (13);
            \draw[bend right] (10) to (14);
            \draw[bend right] (10) to (15);

            \draw[bend left] (11) to (13);
            \draw[bend left] (11) to (14);
            \draw[bend left] (11) to (15);
            \draw[bend left] (11) to (16);

            \draw[bend right] (12) to (14);
            \draw[bend right] (12) to (15);
            \draw[bend right] (12) to (16);
            \draw[bend right] (12) to (17);

            \draw[bend left] (13) to (15);
            \draw[bend left] (13) to (16);
            \draw[bend left] (13) to (17);
            \draw[bend left] (13) to (18);

            \draw[bend right] (14) to (16);
            \draw[bend right] (14) to (17);
            \draw[bend right] (14) to (18);
            \draw[bend right] (14) to (19);

            \draw[bend left] (15) to (17);
            \draw[bend left] (15) to (18);
            \draw[bend left] (15) to (19);
            \draw[bend left] (15) to (20);

            \draw[bend right] (16) to (18);
            \draw[bend right] (16) to (19);
            \draw[bend right] (16) to (20);
            \draw[bend right] (16) to (21);

            \draw[bend left] (17) to (19);
            \draw[bend left] (17) to (20);
            \draw[bend left] (17) to (21);

            \draw[bend right] (18) to (20);
            \draw[bend right] (18) to (21);

            \draw[bend left] (19) to (21);
		\end{tikzpicture}
        \caption{Example of a $k$-tree in the proof of \cref{thm:upper}: $\left\lceil\frac{t+1}{2}\right\rceil$-th power of a path, here illustrated for $t=8$.}
        \label{fig:ksumexample}
    \end{figure}
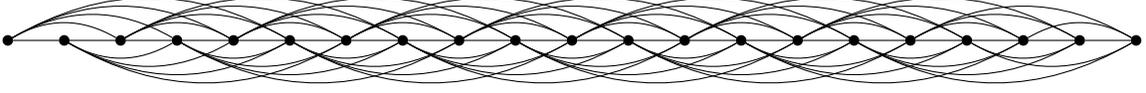

    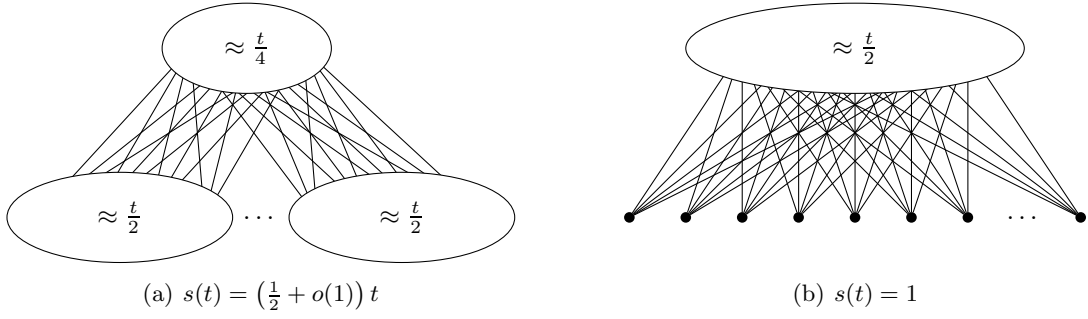
\begin{figure}
     \centering
     \begin{subfigure}[b]{0.49\textwidth}
         \centering
         \begin{tikzpicture}[scale=0.75,
			dot/.style = {circle, fill, minimum size=#1,
			inner sep=0pt, outer sep=0pt},
			dot/.default = 4pt]

            \clip (0,-0.85) rectangle (10,4);
            
            \node(T1) at (3.75,3) {};
            \node(T2) at (4.75,3) {};
            \node(T3) at (5.75,3) {};

            \node[dot] (B1) at (1,0) {};
            \node[dot] (B2) at (2,0) {};
            \node[dot] (B3) at (3,0) {};
            \node[dot] (B4) at (4,0) {};
            \node[dot] (B5) at (6,0) {};
            \node[dot] (B6) at (7,0) {};
            \node[dot] (B7) at (8,0) {};
            \node[dot] (B8) at (9,0) {};
            \node (B9) at (5,0) {\dots};
   
			\draw (T1) to (B1);
            \draw (T1) to (B2);
            \draw (T1) to (B3);
            \draw (T1) to (B4);
            \draw (T1) to (B5);
            \draw (T1) to (B6);
            \draw (T1) to (B7);
            \draw (T1) to (B8);

            \draw (T2) to (B1);
            \draw (T2) to (B2);
            \draw (T2) to (B3);
            \draw (T2) to (B4);
            \draw (T2) to (B5);
            \draw (T2) to (B6);
            \draw (T2) to (B7);
            \draw (T2) to (B8);

            \draw (T3) to (B1);
            \draw (T3) to (B2);
            \draw (T3) to (B3);
            \draw (T3) to (B4);
            \draw (T3) to (B5);
            \draw (T3) to (B6);
            \draw (T3) to (B7);
            \draw (T3) to (B8);

            \filldraw[fill=white] (4.75,3) ellipse (1.5 and 0.8);

            \filldraw[fill=white] (2.5,0) ellipse (2 and 0.8);
            \filldraw[fill=white] (7.5,0) ellipse (2 and 0.8);

            \node at (4.75,3) {$\approx \frac{t}{4}$};
            \node at (2.5,0) {$\approx \frac{t}{2}$};
            \node at (7.5,0) {$\approx \frac{t}{2}$};
		\end{tikzpicture}
         \caption{$s(t)= \left(\frac{1}{2}+o(1)\right)t$}
         \label{subfig:tdiv2}
     \end{subfigure}
     \hfill
     \begin{subfigure}[b]{0.49\textwidth}
         \centering
         \begin{tikzpicture}[scale=0.75,
			dot/.style = {circle, fill, minimum size=#1,
			inner sep=0pt, outer sep=0pt},
			dot/.default = 4pt]

            \clip (0,-0.85) rectangle (10,4);
            
            \node(T1) at (3,3) {};
            \node(T2) at (4,3) {};
            \node(T3) at (5,3) {};
            \node(T4) at (6,3) {};
            \node(T5) at (7,3) {};

            \node[dot] (B1) at (1,0) {};
            \node[dot] (B2) at (2,0) {};
            \node[dot] (B3) at (3,0) {};
            \node[dot] (B4) at (4,0) {};
            \node[dot] (B5) at (5,0) {};
            \node[dot] (B6) at (6,0) {};
            \node[dot] (B7) at (7,0) {};
            \node (B9) at (8,0) {\dots};
            \node[dot] (B10) at (9,0) {};
   
			\draw (T1) to (B1);
            \draw (T1) to (B2);
            \draw (T1) to (B3);
            \draw (T1) to (B4);
            \draw (T1) to (B5);
            \draw (T1) to (B6);
            \draw (T1) to (B7);
            \draw (T1) to (B10);

            \draw (T2) to (B1);
            \draw (T2) to (B2);
            \draw (T2) to (B3);
            \draw (T2) to (B4);
            \draw (T2) to (B5);
            \draw (T2) to (B6);
            \draw (T2) to (B7);
            \draw (T2) to (B10);

            \draw (T3) to (B1);
            \draw (T3) to (B2);
            \draw (T3) to (B3);
            \draw (T3) to (B4);
            \draw (T3) to (B5);
            \draw (T3) to (B6);
            \draw (T3) to (B7);
            \draw (T3) to (B10);

            \draw (T4) to (B1);
            \draw (T4) to (B2);
            \draw (T4) to (B3);
            \draw (T4) to (B4);
            \draw (T4) to (B5);
            \draw (T4) to (B6);
            \draw (T4) to (B7);
            \draw (T4) to (B10);

            \draw (T5) to (B1);
            \draw (T5) to (B2);
            \draw (T5) to (B3);
            \draw (T5) to (B4);
            \draw (T5) to (B5);
            \draw (T5) to (B6);
            \draw (T5) to (B7);
            \draw (T5) to (B10);

            \filldraw[fill=white] (5,3) ellipse (3 and 0.8);

            \node at (5,3) {$\approx \frac{t}{2}$};
		\end{tikzpicture}
         \caption{$s(t)=1$}
         \label{subfig:1}
     \end{subfigure}
        \caption{Examples of graphs $S_{k(s(t),t),r,s(t)}$ in the proof of \cref{thm:upper}.}
        \label{fig:upperexamples}
\end{figure}
    
    \vskip 3mm
    \noindent\textbf{Using $S_{k,r,s}$}

    In order to prove the theorem, we consider the graphs $S_{k(s(t),t),r,s(t)}$ with $k(s(t),t)=\left\lceil\frac{t-s(t)}{2}\right\rceil+1$ and some choice of $s(t)\geq 1$, to be specified later.
    Given that $\lim_{r\rightarrow \infty}\ad(S_{k(s(t),t),r,s(t)})=s(t)-1+2k(s(t),t)=s(t)-1+2\left(\left\lceil\frac{t-s(t)}{2}\right\rceil+1\right)\geq t+1$, we have that $\ad(S_{k(s(t),t),r,s(t)})\geq t$ for sufficiently large $r$.

    Applying \cref{lem:minortosubgraph} and \cref{lem:numberedges}, we only need show that $f(k(s(t),t),s(t),t)=\left(\frac{3}{4}+o(1)\right)\binom{t}{2}$. We show that this holds for various choices of $s(t)$.

    One possible choice for $s(t)$ is $s(t)= \left(\frac{1}{2i}+o(1)\right)t$ for fixed $i\in \N$ (see \cref{subfig:tdiv2} for the case with $i=1$.) Then $k(s(t),t)=\left(\frac{1}{2}-\frac{1}{4i}+o(1)\right)t$. We may then compute that
    \begin{equation*}
        \binom{k(s(t),t)}{2}=\left(\frac{1}{2}-\frac{1}{4i}+o(1)\right)^2\frac{t^2}{2}=\left(\frac{1}{4}-\frac{1}{4i}+\frac{1}{16i^2}+o(1)\right)\binom{t}{2}
    \end{equation*}
    and 
    \begin{equation*}
        k(s(t),t)(t-k(s(t),t))=\left(\frac{1}{2}-\frac{1}{4i}+o(1)\right)\left(1-\left(\frac{1}{2}-\frac{1}{4i}+o(1)\right)\right)t^2=\left(\frac{1}{2}-\frac{1}{8i^2}+o(1)\right)\binom{t}{2}.
    \end{equation*}

    Given that 
    \begin{equation*}
        \left\lfloor\frac{t-k(s(t),t)}{s(t)}\right\rfloor=\left\lfloor\frac{t-\left(\frac{1}{2}-\frac{1}{4i}+o(1)\right)t}{\left(\frac{1}{2i}+o(1)\right)t}\right\rfloor=\left\lfloor\frac{\frac{1}{2}+\frac{1}{4i}+o(1)}{\frac{1}{2i}+o(1)}\right\rfloor=i
    \end{equation*}
    for sufficiently large $t$, we have
    \begin{equation*}
        \left\lfloor\frac{t-k(s(t),t)}{s(t)}\right\rfloor\binom{s(t)}{2}=(i+o(1))\left(\frac{1}{2i}+o(1)\right)^2\frac{t^2}{2}=\left(\frac{1}{4i}+o(1)\right)\binom{t}{2}
    \end{equation*}
    and
    \begin{align*}
        \binom{t-k(s(t),t)-\left\lfloor\frac{t-k(s(t),t)}{s(t)}\right\rfloor s(t)}{2}&=\left(1-\left(\frac{1}{2}-\frac{1}{4i}+o(1)\right)-(i+o(1)) \left(\frac{1}{2i}+o(1)\right)\right)^2\frac{t^2}{2}\\
        &=\left(\frac{1}{16i^2}+o(1)\right)\binom{t}{2}.
    \end{align*}

    Thus we obtain for the value of $f(k(s(t),t),s(t),t)$:  \begin{align*}&\left(\left(\frac{1}{4}-\frac{1}{4i}+\frac{1}{16i^2}+o(1)\right)+\left(\frac{1}{2}-\frac{1}{8i^2}+o(1)\right)+\left(\frac{1}{4i}+o(1)\right)+\left(\frac{1}{16i^2}+o(1)\right)\right)\binom{t}{2}\\&=\left(\frac{3}{4}+o(1)\right)\binom{t}{2},\end{align*} as desired.
    
    Another possible case is $s(t)=o(t)$ (see \cref{subfig:1} for the case $s(t)=1$). An analogous computation to above yields the result in this case (this can informally be seen by letting $i$ tend to infinity). In fact, in this case the result also follows from the approach for $k$-trees discussed above.
\end{proof}

\section{Small graphs}\label{sec:small}

In this section, we prove \cref{thm:small}. We first need the following definitions.

A \emph{(proper) separation} of a graph $G$ is a pair $(A,B)$ such that $A,B\subseteq V(G)$, $A\cup B=V(G)$, $A\setminus B,B\setminus A\neq \emptyset$ and there are no edges between vertices in $A\setminus B$ and $B\setminus A$. The \emph{order} of $(A,B)$ is $|A\cap B|$. We say a graph $G$ is \emph{$k$-connected} if $\v(G)\geq k+1$ and $G$ does not have a separation of order strictly smaller than $k$. Note that complete graphs are the only graphs to not have any separation.

Given a graph $H$ and $k\in \N$, we say a graph $G$ is a \emph{$(H,k)$-cockade} if $G$ is isomorphic to $H$ or if $G$ can be obtained from smaller $(H,k)$-cockades $G'$ and $G''$ by identifying a clique of size $k$ of $G'$ with a clique of size $k$ of $G''$. A simple inductive argument can be used to show that if $G$ is an $(H,k)$-cockade then $\e(G)=\frac{\v(G)-k}{\v(H)-k}\e(H)-\frac{\v(G)-\v(H)}{\v(H)-k}\binom{k}{2}$.

We first prove the following upper bound on the extremal function for minors in $\mathcal K_6^{-4}$, the class of graphs on $6$ vertices and $11$ edges. See, for instance, the introduction of \cite{rolek_extremal_2020}, and references therein, for a summary of similar results on the extremal functions of small graphs. By $K_5^-$ we denote the graph obtained from $K_5$ by removing one edge.

\begin{thm}\label{thm:extremal11}
    If $G$ is a graph such that $\e(G)\geq \frac{5}{2}\v(G)-\frac{7}{2}$, then $G$ contains a minor with $6$ vertices and $11$ edges, unless $G$ is isomorphic to $K_1$, $K_5^-$ or $K_5$.
\end{thm}

\begin{proof}
    First note that $\left\lceil\frac{5}{2}n-\frac{7}{2}\right\rceil=-1, 2, 4, 7, 9, 12$ when, respectively, $n=1,2,3,4,5,6$. It is then immediate that the only graphs $G$ with $\v(G)\leq 5$ and  at least $\frac{5}{2}\v(G)-\frac{7}{2}$ edges are $K_1$, $K_5^-$ and $K_5$. If $\v(G)=6$, then $G$ contains at least $12$ edges, and thus the statement also holds in this case. This shows that the theorem holds if $G$ has at most $6$ vertices, and therefore we may assume $\v(G)\ge 7$.

    Towards a contradiction, suppose then that the statement is false, and let $G$ be a counterexample that minimizes $\v(G)$ and then, subject to $\v(G)$ being minimum, minimizes $\e(G)$. The latter condition in particular implies that $\e(G)=\left\lceil\frac{5}{2}\v(G)-\frac{7}{2} \right\rceil<\frac{5}{2}\v(G)$.
    
    First suppose $G$ is 3-connected. If some $e\in E(G)$ is in fewer than 2 triangles, then $$\e(G/e)\geq \e(G)-2\geq \frac{5}{2}\v(G)-\frac{11}{2}=\frac{5}{2}(\v(G/e)+1)-\frac{11}{2}=\frac{5}{2}\v(G/e)-3.$$
    By minimality of $G$, $G/e$, contains a minor on 6 vertices and 11 edges (using that $\v(G/e)\geq 6$ to exclude the small exceptions). As this contradicts our assumptions on $G$, every edge in $G$ lies on at least two triangles.

    Next suppose there exists $A\subseteq V(G)$ of size $5$ such that $\e(G[A])=8$. Let $u\in V(G)\setminus A$. As $G$ is 3-connected, Menger's theorem \cite{menger_zur_1927} implies there exist internally vertex-disjoint $u-A$ paths $P_1,P_2,P_3$ in $G$ (with no internal vertices in $A$). Then, $G[A\cup V(P_1)\cup V(P_2)\cup V(P_3)]$ contains a minor on 6 vertices and at least 11 edges, which can be obtained by contracting all but one edge in each of $P_1,P_2,P_3$. Hence such a set $A$ does not exist.

    Given that $\ad(G)=\frac{2\e(G)}{\v(G)}<5$, there exists $u\in V(G)$ of degree at most 4 (and at least 3, by 3-connectivity).

    Consider the case $\deg_G(u)=4$. As every edge of $G$ is in at least 2 triangles, $\delta(G[N_G[u]])\geq 3$. In particular, $\e(G[N_G[u]])=\frac{1}{2}\ad(G[N_G[u]])\v(G[N_G[u]])\geq \frac{15}{2}$, and as this is an integer $\e(G[N_G[u]])\geq 8$. However, we have excluded such a choice $A=N_G[u]$ earlier.

    Hence, we may suppose that $\deg_G(u)=3$, say $N_G(u)=\{x,y,z\}$. Again as every edge of $G$ is in at least 2 triangles, $G[N_G[u]]$ is necessarily isomorphic to $K_4$. Let $v\in N_G(x)\setminus N[u]$ (such a vertex necessarily exists as otherwise $(\{u,x,y,z\},V(G)\setminus\{u,x\})$ would form a separation of order $2$ in $G$, contradicting 3-connectivity). As proved above, $vx$ is in at least one triangle. If $v$ is adjacent to $y$ or $z$, then $G[\{u,x,y,z,v\}]$ contains at least 8 edges, which we have excluded earlier. Hence, there exists $w\in V(G)\setminus N_G[u]$ which is adjacent to both $x$ and $v$. Again, by Menger's theorem there exist at least three internally vertex-disjoint paths between $\{v,w\}$ and $\{u,y,z\}$. At most one of these can contain $x$. Let $P_1,P_2$ be two of these paths which don't contain $x$, we may also suppose they do not contain any of $\{u,x,y,z,v,w\}$ as internal vertices. Then, $G[\{u,x,y,z,v,w\}\cup V(P_1)\cup V(P_2)]$ contains a minor on 6 vertices and 11 edges, which can be obtained by contracting all but one edge in each of $P_1,P_2$.

    Hence, we may now suppose that $G$ is not 3-connected. As $\v(G)\geq 7$, we may then suppose $G$ is not a complete graph, so let $(A,B)$ be a separation of $G$ of minimal order. In particular, $|A\cap B|\leq 2$.

    We divide the rest of the proof into cases depending on size of $A \cap B$.
    First suppose $|A\cap B|=0$. By minimality of $G$, if $\e(G[A])\geq \frac{5}{2}|A|-2$ (in particular, $G[A]$ cannot be isomorphic to $K_1, K_5^-$ or $K_5$), then $G[A]$ contains a minor on $6$ vertices and 11 edges, which is a contradiction. Hence, we may assume that $\e(G[A])< \frac{5}{2}|A|-2$, and similarly that $\e(G[B])< \frac{5}{2}|A|-2$. Then,
    $$\e(G)=\e(G[A])+\e(G[B])<\frac{5}{2}(|A|+|B|)-4=\frac{5}{2}\v(G)-4<\frac{5}{2}\v(G)-\frac{7}{2},$$
    which is a contradiction to our hypothesis, so this case is not possible.

    Now suppose $|A\cap B|=1$, say $A\cap B=\{x\}$. Then $G$ is connected, and in particular $G[A]$ and $G[B]$ are connected and $|A|,|B|\ge 2$. If $G[A]$ is isomorphic to $K_5$, then $G[A\cup y]$ is a 6-vertex graph with at least 11 edges, where $y$ is a neighbour of $x$ in $B$. Hence, $G[A]$ is not isomorphic to $K_5$, and similarly for $G[B]$. Then, by minimality of $G$, if $\e(G[A])\geq \frac{5}{2}|A|-3$ (note that this implies that $A$ cannot isomorphic to $K_5^-$), $G[A]$ contains a minor on 6 vertices and 11 edges, which is a contradiction. Hence, we may assume that $\e(G[A])< \frac{5}{2}|A|-3$, and similarly that $\e(G[B])< \frac{5}{2}|B|-3$. Then,
    $$\e(G)=\e(G[A])+\e(G[B])<\frac{5}{2}(|A|+|B|)-6=\frac{5}{2}(\v(G)+1)-6=\frac{5}{2}\v(G)-\frac{7}{2},$$
    which contradicts our hypothesis, so this case is not possible.

    Finally, suppose $|A\cap B|=2$; say $A\cap B=\{x,y\}$. If $x,y$ are in different components of $G[A]-xy$, at least one of $x$ or $y$ would be a cut vertex, which would contradict our choice of separation. Hence, there exists an $x-y$ path $P_1$ with at least 2 edges and with internal vertices in $A\setminus B$. Similarly, there exists an $x-y$ path $P_2$ with at least 2 edges and with internal vertices in $B\setminus A$.

    If $G[A]$ is isomorphic to $K_5^-$ or $K_5$, then $G[A\cup V(P_2)]$ contains a minor on 6 vertices and at least 11 edges, which we can obtain by contracting all but 2 of the edges of $P_2$. Hence, $G[A]$, and similarly $G[B]$, are not isomorphic to $K_5^-$ or $K_5$.

    By minimality of $G$, if $\e(G[A])\geq \frac{5}{2}|A|-\frac{7}{2}$ or $\e(G[B])\geq \frac{5}{2}|B|-\frac{7}{2}$, then $G$ contains a minor on $6$ vertices and $11$ edges, which is a contradiction. Given that the number of edges must be an integer, we have that $\e(G[A])\leq \frac{5}{2}|A|-4$ and $\e(G[B])\leq \frac{5}{2}|B|-4$.

    If $xy\in E(G)$, then
    $$\e(G)=\e(G[A])+\e(G[B])-1\leq \frac{5}{2}(|A|+|B|)-9=\frac{5}{2}(\v(G)+2)-9<\frac{5}{2}\v(G)-\frac{7}{2},$$which is a contradiction to our hypothesis.
    
    Hence, $xy\notin E(G)$. Suppose first that $\e(G[A])\leq \frac{5}{2}|A|-\frac{9}{2}$ and  $\e(G[B])\leq \frac{5}{2}|B|-\frac{9}{2}$. Then,
    $$\e(G)=\e(G[A])+\e(G[B])\leq \frac{5}{2}(|A|+|B|)-9=\frac{5}{2}(\v(G)+2)-9<\frac{5}{2}\v(G)-\frac{7}{2},$$which would be a contradiction to our assumptions on $G$. Thus, at least one of the two above inequalities is invalid. Without loss of generality, we may assume that $\e(G[A])> \frac{5}{2}|A|-\frac{9}{2}$ and thus $\e(G[A])\ge \frac{5}{2}|A|-4$. Now this implies $\e(G[A]+xy)=\e(G[A])+1\ge \frac{5}{2}|A|-3>\frac{5}{2}|A|-\frac{7}{2}$. Contracting $P_2$ into an edge, we obtain that $G[A]+xy$ is a minor of $G$. Using the minimality of $G$, we then find that $G[A]+xy$ is isomorphic to $K_5^-$ or $K_5$. The case $G[A]+xy\simeq K_5^-$ is impossible, as it would mean that $\e(G[A])=8<\frac{5}{2}|A|-4$. Thus, we have $G[A]+xy \simeq K_5$. But then by removing superflous vertices and edges from $B$ and contracting all but two of the edges in $P_2$, we obtain a minor of $G$ isomorphic to the graph obtained from $G[A]\simeq K_5^-$ by adding a new vertex adjacent to $x$ and $y$. This is a graph on $6$ vertices and $11$ edges, as desired.
    
    As we have found a contradiction to our initial assumption that $G$ is a smallest counterexample in every possible case, this completes the proof of the theorem. 
\end{proof}

Although \cref{thm:extremal11} is sufficiently strong for our purposes, it might be possible to improve it, as we are not aware of any family of graphs $G$ with no $\mathcal K_6^{-4}$ minor for which $\e(G)\approx\frac{5}{2}\v(G)$. However, $(K_5^-,1)$-cockades do not contain any $\mathcal K_6^{-4}$ minor and contain $\approx \frac{9}{4}\v(G)$ edges. 

We are now ready to prove \cref{thm:small}, which we restate for convenience.

\smallthm*

\begin{proof}\item

\vskip 2pt
\noindent{\underline{\textbf{Upper bounds}}}
\vskip 5pt 

We show that the values in the statement cannot be improved. For $t=2,3$, these values cannot be improved as no graph on $t$ vertices can have more than $\binom{t}{2}$ edges.

    For $t=4,5$, consider the graphs $S_{2,r,t-3}$ as defined in \cref{sec:upper}. One easily verifies that for such $t$ and $r\geq 4$, $$\ad(S_{2,r,t-3})=\frac{2\left(\binom{2}{2}+r\binom{t-3}{2}+2 r (t-3)\right)}{2+r(t-3)}\geq t-1.$$ By \cref{lem:minortosubgraph} and \cref{lem:numberedges}, $S_{2,r,t-3}$ does not contain any minor on $t$ vertices with more than $f(2,t-3,t)$ vertices, which we can directly compute to be $5$ and $8$ for, respectively, $t=4$ and $t=5$.
    
    For $t=6$, consider any $(K_5^-,2)$-cockade $G$ with $\v(G)\geq 26$. First note that $\e(G)=\frac{\v(G)-2}{5-2}\cdot 9-\frac{\v(G)-5}{5-2}\binom{2}{2}=\frac{8}{3}\v(G)-\frac{13}{3}\geq \frac{5}{2}\v(G)$ and so $\ad(G)\geq 5$. However, we claim that $G$ cannot any minor on $6$ vertices and $12$ edges.
 
    It is easy to see that every graph on $6$ vertices with at least $12$ edges is either $3$-connected or contains a $K_5$ subgraph. 
    However, as $G$ is constructed in a tree-like fashion by identifying edges between copies of $K_5^-$, the only $3$-connected minors of $G$ are in fact minors of $K_5^-$ (more generally, for $k=0,1,2$, it is well-known that if $G$ is a $k$-sum of $G_1$ and $G_2$, and $H$ is $3$-connected, then $H$ is a minor of $G$ if and only if $H$ is a minor of $G_1$ or of $G_2$). Hence, $G$ can contain neither a $3$-connected $6$-vertex graph nor $K_5$ as a minor, proving our claim


    \vskip 5pt 
    \noindent{\underline{\textbf{Lower bounds}}}

    \vskip 5pt 
    \noindent{\boldmath$t=2$ \textbf{:}} Any graph with average degree at least one contains an edge, and thus contains a minor on 2 vertices with one edge.

    \vskip 5pt 
    \noindent{\boldmath$t=3$ \textbf{:}} It is well-known that if $G$ is a forest, $\e(G) \leq \v(G)-1$, wand in particular that $\ad(G)=\frac{2\e(G)}{\v(G)}<2$ (if $G$ is non-null). Given that $\ad(G)\geq 3-1=2$, $G$ is not a forest and thus contains a cycle $C$ and thus $G$ contains $K_3$  as a minor.

    \vskip 5pt 
    \noindent{\boldmath$t=4$ \textbf{:}} It is well-known, and easy to see, that $K^{-}_4$-minor-free graphs are the graphs for which every component is a \emph{cactus graph}, that is a connected graph in which every block is either an edge or a cycle. It is also well-known (for instance, \cite[Exercice 4.1.31]{west_introduction_2001}) that if $G$ is a cactus graph, then $\e(G)\leq\left\lfloor\frac{3(\v(G)-1)}{2}\right\rfloor$ edges (the proof proceeds by induction on the number of blocks). Given that $G$ has average degree at least $3$, $\e(G)\ge \frac{3}{2}\v(G)$ and so $G$ contains a $K^{-}_4$ minor, i.e. a minor on four vertices with five edges, as claimed.
    
    \vskip 5pt 
    \noindent{\boldmath$t=5$ \textbf{:}} Dirac \cite[Theorem 1B]{dirac_homomorphism_1964} proved that for any graph $G$ such that $\e(G)\geq 2\v(G)-2$, either $G$ contains a minor on $5$ vertices and $8$ edges or $G$ is a $(K_4,1)$-cockade. Note that in the latter case, $\e(G)=\frac{\v(G)-1}{4-1}\cdot 6-\frac{\v(G)-4}{4-1}\binom{1}{2}=2(\v(G)-1)<2\v(G)$. Hence, if $\ad(G)\geq 4$, we have that $\e(G)\geq 2\v(G)$ and so $G$ contains a minor on $5$ vertices and $8$ edges.

    \vskip 5pt 
    \noindent{\boldmath$t=6$ \textbf{:}} Given that $\d(G)\geq \frac{5}{2}\d(G)$, \cref{thm:extremal11} implies this result directly (noting that none of the small exceptions in \cref{thm:extremal11} have average degree at least $5$).
\end{proof}

\section{Concluding remarks}

We have considered the problem of finding the best possible $\alpha$ such that every graph with average degree at least $t$ contains a minor on $t$ vertices with at least $\left(\alpha -o(1)\right)\binom{t}{2}$ edges; we have shown that $\sqrt{2}-1\leq \alpha\leq \frac{3}{4}$. It would be interesting to further improve these bounds.

We note that in our proof of \cref{thm:main}, we have only used contractions to be able to consider the smallest minor of $G$ such that $\ad(G)\geq t$ and apply \cref{lem:neighbourhoodmindegree}. Once we have obtained that all closed neighbourhoods have minimum degree greater than $\frac{t}{2}$, we only consider subgraphs. In this setup, we cannot improve our lower bound on $\alpha$ beyond $\frac{1}{2}$. Indeed, consider the line graph of the complete graph $K_n$. It is  $t:=2(n-2)$-regular, has closed neighbourhoods of minimum degree $(n-1)>\frac{t}{2}$ and it is not hard to verify that this graph contains no subgraph on $t$ vertices with more than $(1+o(1))n^2=(\frac{1}{2}+o(1))\binom{t}{2}$ edges.

\section*{Acknowledgments}
This research was partially completed at the Second 2022 Barbados Graph Theory Workshop held at the Bellairs Research Institute in December 2022.

\bibliography{refs}
\bibliographystyle{abbrvurl}

\end{document}